\documentclass{amsart}
\usepackage[utf8]{inputenc}
\usepackage{amsmath}
\usepackage{amsfonts}
\usepackage{mathrsfs}
\usepackage{amsthm}
\usepackage{amssymb}
\usepackage{graphicx}
\usepackage[colorlinks=true, allcolors=blue]{hyperref}

\numberwithin{equation}{section}

\newtheorem{theorem}{Theorem}[section]

\newtheorem{corollary}[theorem]{Corollary}
\newtheorem{lemma}[theorem]{Lemma}
\newtheorem{proposition}[theorem]{Proposition}
\newtheorem{remark}[theorem]{Remark}

\title[Rigidity results for MCF graphical translators]{Rigidity results for mean curvature flow graphical translators moving in non-graphical direction}

\author{John Man Shun Ma}
\address{Department of Mathematical Sciences, University of Copenhagen, Universitetsparken 5, DK-2100 Copenhagen Ø, Denmark}
\email{jm@math.ku.dk}

\author{Yuan Shyong Ooi}
\address{Department of Mathematics, Pusan National University, Busan 46241, Korea}
\email{yuanshyong@pusan.ac.kr}

\author{Juncheol Pyo}
\address{Department of Mathematics, Pusan National University, Busan 46241, Korea}
\email{jcpyo@pusan.ac.kr}


\date{}

\begin{document}
\maketitle
\begin{abstract}
    In this paper, we study the rigidity results of complete graphical translating hypersurfaces when the translating direction is not in the graphical direction. We proved that any entire graphical translating surface in the translating direction not parallel to the graphical one is flat if either the translating surface is mean convex or the entropy of the translating surface is smaller than $2$. For higher dimensional case, we show that the same conclusion holds if the graphical translating hypersurface satisfies certain growth condition.  
\end{abstract} 

\section{Introduction}
A smooth hypersurface $\Sigma^n$ immersed in $\mathbb{R}^{n+1}$ is called a translating soliton (or translator) if its mean curvature $H$ satisfies the equation  
\begin{equation}\label{translator eqn}
H=\langle\nu,T\rangle
\end{equation}
where $\nu$ is the unit normal of $\Sigma$ and $T$ is any constant unit vector in $\mathbb{R}^{n+1}$. 

In the study of Type II singularity formation of mean curvature flow (MCF), translating soliton arises as one of the singularity model after parabolic rescaling the flow. In term of MCF, any translating soliton gives us a translating solution $\Sigma_s=\Sigma+sT$ for any $s\in\mathbb{R}$. This means under MCF, the shape of a translating soliton remains unchange but only moves in the $T$ direction. On the other hand, translating soliton can also be viewed as a minimal hypersurface in $(\mathbb{R}^{n+1},\bar{g})$ where $\bar{g}$ is a conformally flat metric \cite{ilmanen1994elliptic}. Therefore, we may expect that translator shares some analog properties as minimal hypersurface in Euclidean space.

We say that $\Sigma$ is a \emph{graphical translator} if it can be expressed as a graph of a smooth function $u:\Omega\subseteq\mathbb{R}^n\rightarrow\mathbb{R}$. In this setting, if we write $T=(T',T_{n+1})$ where $T'\in\mathbb{R}^n$ and
$\Sigma=\{(x,u(x)):x\in\Omega\subseteq\mathbb{R}^n\}$ then the function $u$ satisfies the following quasilinear elliptic PDE in divergence form
\begin{equation}
    \sum_{i=1}^nD_i\left(\frac{D_iu}{\sqrt{1+|Du|^2}}\right)=\frac{-\langle Du,T'\rangle+T_{n+1}}{\sqrt{1+|Du|^2}}
\end{equation}
Note that in the graphical case, we always choose the orientation as upward unit normal $\nu=\frac{(-Du,1)}{\sqrt{1+|Du|^2}}$. We say that $\Sigma$ is an \emph{entire graph} if its domain $\Omega=\mathbb{R}^n$.

Let us denote $\{e_i\}$ as the standard basis in $\mathbb{R}^{n+1}$. The case for $T=e_{n+1}$ is well studied in the graphical setting. We shall call this type of translator a vertical graphical translator. From the translator equation (\ref{translator eqn}), it is clear that graphical condition implies strictly mean-convexity $H=\langle \nu,e_{n+1}\rangle>0$. So graphical translator is always (strictly) mean-convex when the translating direction is upward. 

Now let us give some examples of vertical graphical translator. When $n=1$, the only complete graphical translating curve in the plane is the grim reaper $\Gamma_{gr}$ i.e. the graph of the function
\begin{equation}
    y=-\log\cos x, \quad x\in\left(\frac{-\pi}{2},\frac{\pi}{2}\right)
\end{equation}
We can also produce new translator by taking Cartesian product of a translator with euclidean factor. For example the Cartesian product of $\Gamma_{gr}$ with euclidean factor will give us a graphical translator $\Gamma_{gr}\times\mathbb{R}$ known as the grim reaper surface. By rescaling and the rotating grim reaper surface, we can produce a family of graphical translator known as the tilted grim reaper surface $\{\Sigma^\theta_{tgr}\}_{\theta\in(0,\pi/2)}$ i.e.
\begin{equation}
    \Sigma^\theta_{tgr}:=\left\{\left(x,y,-\frac{\log\cos(x\cos\theta)}{\cos^2\theta}-y\tan\theta\right):x\in\left(-\frac{\pi}{2\cos\theta},\frac{\pi}{2\cos\theta}\right),y\in\mathbb{R}\right\}
\end{equation}
Note that the family of tilted grim reaper surface is defined on a slab of width $\frac{\pi}{\cos\theta}$ which range from $\pi$ to $\infty$ as $\theta$ goes from $0$ to $\frac{\pi}{2}$.

Besides the tilted grim reaper, another example of family of graphical translator defined on a slab, known as the $\triangle$-wing, is constructed by Hoffman et al. \cite{HIMW2019}. They show that for each $b>\pi/2$, there exist a unique, complete, strictly convex graphical translator $u^b:(-b,b)\times\mathbb{R}\rightarrow\mathbb{R}$. Note that this family of translating graph is different from tilted grim reaper surface because tilted grim reaper surface is not strictly convex.

In \cite{clutterbuck2007stability}, Clutterbuck-Sch\"{u}rer-Schulze constructed the unique entire, rotationally symmetry, vertical graphical translator known as the bowl soliton (see also \cite{altschuler1994translating}). The bowl soliton is not only mean-convex but is strictly convex and is asymptotic to a paraboloid.

For higher dimension, any Cartesian product of complete minimal hypersurface with a Euclidean factor will give us a translator with translating direction parallel to the Euclidean factor direction. In particular we can take the non-planar entire graphical solution of minimal surface equation $\Sigma^8\subseteq\mathbb{R}^9$ constructed by Bombieri et al. \cite{bombieri1969minimal} and product it with $\mathbb{R}$. $\Sigma\times\mathbb{R}$ then gives us a non-planar, entire graphical translator in $\mathbb{R}^{10}$ with horizontal translating direction. Note that this example of translator is static under MCF since it is also a minimal hypersurface.

There are also various other example of translator which are not graphical. To name a few, Clutterbuck et al. \cite{clutterbuck2007stability} construct the winglike translator which can be viewed as the union of two bowl soliton with a neck; Nguyen \cite{nguyen2009translating} uses gluing construction to construct the translating tridents; Kim-Pyo \cite{kim2018existence} study helicoidal type of translating soliton and completely classified all of them; Hoffman, Martin and White \cite{hoffman2022nguyen} construct and classify semi-graph translating soliton.

We can now state the known classification result for two dimensional complete vertical graphical translator. First of all, by the work of Shahriyari \cite{shahriyari2015}, the domain of any complete, two dimensional, vertical translating graph can only be $\mathbb{R}^2$, half-plane or slab $(-b,b)\times\mathbb{R}$. Building upon the work of Shariyari, Spruck-Xiao and Wang\cite{spruckxiao2020,wang2011}, Hoffmann et.al \cite{HIMW2019} classify all the vertical translating graph as grim reaper surface, tilted grim reaper surface, bowl soliton and $\triangle$-wing surface. Among them, only bowl soliton is the entire example, the others are all defined on a slab. They also show that there does not exist any complete vertical graphical translator whose domain is a half plane.

For higher codimensional translator related result, one can refer to the work of Xin and Kunikawa \cite{Xin2015,kunikawa2015,kunikawa2017}. 

We are interested in the rigidity result of complete graphical translator which is not translating in the vertical direction. Our first observation is, if the domain $\Omega\subset\mathbb{R}^n$ is bounded and the velocity $T$ is not vertical then by applying the half space type result of Kim-Pyo \cite{kim2021half}, we know that such graphical translator cannot exist. Therefore we only need to study the unbounded domain cases. In particular we shall focus on the entire domain case. By assuming mean-convexity condition, we are able to show that entire, mean convex graphical translator not translating vertically must be a hyperplane: 

\begin{theorem}\label{mean convex}
Let $\Sigma^2=\{(x,y,u(x,y)):\forall x,y\in\mathbb{R}\}\subseteq\mathbb{R}^3$ for some $u\in C^\infty(\mathbb{R}^2)$ be an entire graphical surface satisfying the translator equation 
\[H=\langle\nu,T\rangle\]
where $T\in\mathbb{R}^3$ is a unit vector not parallel to $e_3$ and $\nu$ is upward unit normal. If $H\geq 0$ then $\Sigma$ is a plane parallel to $T$ direction.
\end{theorem}

In general, dimension restriction is needed since we could have non-planar example in higher dimension. (see Section \ref{section counter example} for counter example)

On the other hand, Hershkovits \cite{hershkovits2020translators} classified translator with entropy less than or equal to the entropy of a cylinder. In particular he shows that if a translating surface $\Sigma^2$ satisfies the entropy bound
\[
\lambda(\Sigma)\leq \sqrt{\frac{2\pi}{e}}\approx 1.52
\]
then it can only be a plane with entropy 1 or the bowl soliton with entropy $\sqrt{\frac{2\pi}{e}}$. If $\Sigma$ lies in a slab and is simply connected, Chini \cite{chini2020simply} can classify those translators with entropy less than 3. He showed that vertical plane, tilted grim reaper surface, grim reaper surface or $\triangle$-wing surface are the only possible examples. In our case, we can classify all complete non-vertical graphical translator with entropy less than 2:

\begin{theorem} \label{thm rigidity in n=2 assuming lambda <2}
Let $\Omega$ be an open subset of $\mathbb R^2$ and let $\Sigma= \{ (x, y, u(x, y)) : (x, y)\in \Omega\}$ be a complete graphical translator in the direction $T$ not parallel to $e_3$. If $\lambda (\Sigma)<2$, $\Sigma$ is a hyperplane. 
\end{theorem}

Another way to look at the rigidity problem for translator is to study their Gauss map image. In this direction, Bao-Shi \cite{bao2014gauss} studied translating soliton whose image of Gauss maps are contained in compact subsets of upper hemisphere of the standard $\mathbb{S}^n$. Their result says that such translator can only be a hyperplane. For graphical setting, image of Gauss map is closely related to the asypmtotic behaviour of the gradient of the graph. In particular, any entire graphical translator with bounded gradient must have Gauss images contained in compact subsets of upper hemisphere and hence must be a hyperplane. If we assume that $u$ satisfies certain growth rate, we can show the following rigidity result:

\begin{theorem} \label{thm main theorem}
Assume that $T$ is perpendicular to $e_{n+1}$, and let $\Sigma$ be a graphical entire translator with velocity $T$ given by $\Sigma = \{ (x, u(x)) : x\in \mathbb R^n\}$. If there are positive numbers $C_1, C_2$ so that 
\begin{equation} 
|u(q)| \le C_1  + C_2\sqrt{ |\langle q,T\rangle|}, \ \ \ \forall q\in \mathbb R^n
\end{equation}
then $\Sigma$ is the stationary horizontal plane. 
\end{theorem}

By assuming a slow growth rate of the gradient, we can also obtain the following result:
\begin{theorem}\label{slow gradient growth}
Let $\Sigma$ be an entire graphical translator given by $\Sigma = \{ (x, u(x)) : x\in \mathbb R^n\}$ with unit velocity $T$ not parallel to $e_{n+1}$ direction. If the gradient of $u$ satisfies the growth rate 
\begin{equation}
    |Du(x)|=o(|x|^{1/4})
\end{equation}
then $\Sigma$ is a hyperplane. 
\end{theorem}

Recently, Gama, Martin and Møller \cite{gama2022finite} studied non-vertical graphical translator in $\mathbb{R}^3$ that lies between two parallel planes in the translating direction. The conditions they impose on are the entropy and the width of the two parallel planes i.e. the distance between the two planes. They can show that if the width is finite and if the translating direction is not horizontal then the graphical translating surface can only be a plane \cite[Proposition 3.13]{gama2022finite}. We can think of finite width condition as the graph function has at most linear growth. Besides that, they also classify those simply connected translator $\Sigma^2$ with finite width and with entropy satisfying $3\leq\lambda(\Sigma)<4$. Although we impose stronger growth rate condition in Theorem \ref{thm main theorem}, we consider the rigidity result of their complement case which is when the translating direction is horizontal and our method also works for any dimension. 

The organization of this paper is as follows. We provide some background needed in the proofs of our results in Sec. \ref{Prelimenary}. In Sec. \ref{mean convex graphical}, we give a prove for Theorem \ref{mean convex}. The proofs for Theorem \ref{thm rigidity in n=2 assuming lambda <2} is given in Sec. \ref{section small entropy}. In Sec. \ref{section graphical with growth rate} we give a proof to Theorem \ref{thm main theorem} and Theorem \ref{slow gradient growth}. The final section is to provide counter a example for our results in higher dimension showing that the dimension restriction for the first two results is optimal. 

\section{Prelimenary}\label{Prelimenary}
\subsection{Colding-Minicozzi's entropy}

Let $\Sigma^{n}\subseteq\mathbb{R}^{n+1}$ be any hypersurface. Following the work of Colding-Minicozzi \cite{colding2012generic},  given $x_0\in\mathbb{R}^{n+1}$ and $t_0>0$, we define the F-functional $F_{x_0,t_0}$ (see also \cite{huisken1990asymptotic}, \cite{angenent1995computed}) by
\begin{equation}
    F_{x_0,t_0}(\Sigma):=(4\pi t_0)^{-\frac{n}{2}}\int_\Sigma e^{-\frac{|x-x_0|^2}{4t_0}}d\mu
\end{equation}
and its entropy functional $\lambda=\lambda(\Sigma)$ given by 
\begin{equation}
    \lambda=\sup_{x_0,t_0}F_{x_0,t_0}(\Sigma)
\end{equation}
In \cite{colding2012generic}, the F-functional is used by Colding-Minicozzi to study the singularities of MCF. The key property about this functional is that its critical point is precisely when $\Sigma$ is a $t=-t_0$ slice of self-shrinking solution that becomes extinct at $x_0$ and $t=0$.

The reason to work with entropy is as follow. For any properly embedded smooth hypersurface $\Sigma$ in $\mathbb{R}^{n+1}$, we have $\lambda(\Sigma)\geq 1$ and it is invariant under dilations and rigid motions of $\Sigma$. Moreover $\lambda(\Sigma)=1$ exactly when $\Sigma$ is a flat hyperplane. 

In particular, when $\Sigma_t=\Sigma+t T $ is a MCF translating solution with $\Sigma_0=\Sigma$ with finite entropy, and $\Sigma^\delta_t:=\delta\Sigma_{\delta^{-2}t}$ is any parabolic re-scaling of $\Sigma_t$, by translation and dilation invariant properties of $\lambda(\Sigma)$ all their entropy is preserved i.e. 
\begin{equation} \label{eqn entropy is preserved under parabolic rescaling}
\lambda(\Sigma^\delta_t)=\lambda(\Sigma_t)=\lambda(\Sigma)
\end{equation}

Finite entropy also gives us some control on the area of hypersurface inside an extrinsic ball in the following sense,
\begin{lemma} \label{lemma entropy bound implies area growth}
Let $\Sigma$ be any smooth hypersurface in $\mathbb{R}^{n+1}$. If $\lambda(\Sigma)$ is finite then $\Sigma$ has extrinsic Euclidean area growth. This means for any $R>1$, there exist a constant $C=C(n)>0$ such that
\[
\text{Area}(\Sigma\cap B^{n+1}_R(0))\leq C\dot\lambda(\Sigma)R^n
\]
\end{lemma}
\begin{proof}
From the definition of entropy, by choosing $t_0=R^2$ and $x_0=0$
\[
\lambda(\Sigma)\geq (4\pi R^2)^{-n/2}\int_{\Sigma\cap B_R}e^{-|x|^2/4R^2}\geq e^{-1/4}(4\pi)^{-n/2}R^{-n}\text{Area}(\Sigma\cap B_R)
\]
Hence
\begin{equation}
    \text{Area}(\Sigma\cap B_R)\leq e^{1/4}(4\pi)^{n/2}\lambda(\Sigma)R^n
\end{equation}
where our constant $C(n)=e^{1/4}(4\pi)^{n/2}$
\end{proof}

We also recall the definition of a blow-down of a translator with finite entropy. By Lemma \ref{lemma entropy bound implies area growth} and (\ref{eqn entropy is preserved under parabolic rescaling}), we have
\begin{equation}
    \text{Area}(\Sigma^\delta_t \cap B_R)\leq C\lambda(\Sigma)R^2
\end{equation}
for all $t<0$ and $\delta >0$. Arguing as in \cite{ilmanensing}, by the compactness theorem on Brakke flow \cite[7.1]{ilmanen1994elliptic}, there is a sequence $\delta_i \to 0$ and a limit Brakke flow $\{\nu_t\}_{t<0}$ so that $M^{\lambda_i}_t$ converges weakly to $\nu_t$ for a.e. $t$. 
Using Huisken's Monotonicity formula, we have for all $a<b<0$, 
\begin{equation*}
\int_a^b \int_{M^\lambda_t} \left| \vec H(x) + \frac{x^\perp }{-2t}\right|^2\Phi (x, t) d\mu_t^\lambda dt \to 0
\end{equation*}
as $\lambda \to 0$, where 
\begin{align*}
\Phi(x, t) = \frac{1}{(-4\pi t)^{\frac{n}{2}}} e^{-\frac{|x|^2}{-4t}}. 
\end{align*}
From there we see that the limit $\{\nu_t\}$ is self-similar: 
$$ \vec H (x) = -\frac{x^\perp}{2t}, \ \ \ \forall t <0.$$
For surfaces in $\mathbb R^3$, one obtains better regularity on $\{ \nu_t\}$: Using the local Gauss-Bonnet Estimate \cite[Theorem 3]{ilmanensing}, it is proved that \cite[Theorem 1,2]{ilmanensing} the support of $\{v_t\}$ is embedded: that is, $|v_t| = \sqrt{-t} \Sigma_\infty$ and $\Sigma_\infty$ is an embedded self-shrinker. Also, there is $\tau_i\to- 1$ so that $\Sigma^{\delta_i}_{\tau_i}$ converges locally smoothly (possibly with multiplicity) as $i\to +\infty$ to $\Sigma_\infty$ away from a discrete set in $\mathbb R^3$. Moreover, $g(\Sigma_\infty) \le g(\Sigma)$ (In \cite{ilmanensing}, the theorem is stated for blow-ups, i.e. $\delta \to \infty$, but the same argument also works for blow-down under the finite entropy assumption).

%With this area bound, we can obtain compactness result for translating surfaces with uniformly bounded entropy and genus using Brian White's compactness result (c.f \cite{white2018compactness}).
\subsection{Properties of complete graphical translator}
It is known that complete vertical graphical translator cannot have bounded domain \cite{shahriyari2015}. We show the same result holds true if the translating direction is not vertical. 
\begin{proposition}
Let $\Omega\subset\mathbb{R}^n$ be a bounded domain and 
\[
\Sigma:=\{(x,u(x)):\forall x\in\Omega\}\subseteq\mathbb{R}^{n+1}
\]
be a complete graphical translator with translating direction $T$ not parallel to $e_{n+1}$. Then such $\Sigma$ does not exist. 
\end{proposition}
\begin{proof}
By a rotation of $SO(n+1)$ fixing the $e_{n+1}$ direction, we may assume that $T=\cos\theta e_{1}+\sin\theta e_{n+1}$ for some $\theta\in(-\pi/2,\pi.2)$. We can also assume that the domain $\Omega\subseteq\{x\in\mathbb{R}^n: a\leq x_1\leq b\}$ for some finite number $a,b$. If $\Sigma$ exists, it must lie in the half space $H=\{x\in\mathbb{R}^{n+1}: \langle x,e_1\rangle\leq b\}$. Since $\langle e_1,T\rangle=\cos\theta>0$, by half-space result for translator of Kim and Pyo \cite[Theorem 1.1]{kim2021half}, we know that $\Sigma$ cannot exist. 
\end{proof}

\section{Mean convex graphical translator}\label{mean convex graphical}
In this section, we will prove the rigidity result for entire mean convex non-vertical graphical translator. We show that plane parallel to the translating direction is the only possible case under this assumption. As a corollary, we obtain the complete classification of mean-convex graphical translator in any translating direction. 
\begin{theorem}[Theorem \ref{mean convex}]
Let $\Sigma^2=\{(x,y,u(x,y)):\forall x,y\in\mathbb{R}\}\subseteq\mathbb{R}^3$ for some $u\in C^\infty(\mathbb{R}^2)$ be an entire graphical surface satisfying the translator equation $H=\langle\nu,T\rangle$ where $T\in\mathbb{R}^3$ is a unit vector not parallel to $e_3$ and $\nu$ is unit upward normal. If $H\geq 0$ then $\Sigma$ is a plane parallel to $T$ direction.
\end{theorem}

\begin{proof}
First of all, recall the structural equation for mean curvature of a translating soliton (see for example \cite[Lemma 2.1(f)]{martin2015}):
\begin{equation}
    \Delta H+\langle T,\nabla H\rangle+|A|^2H=0\
\end{equation}
Since $H\geq 0$ we have $\Delta H+\langle T,\nabla H\rangle\leq 0$. By strong maximum principle, either $H\equiv 0$ or $H>0$.

If $H\equiv 0$ then $\langle\nu,T\rangle=0$ on $\Sigma$, hence $T$ is a parallel constant tangential vector field on $\Sigma$. The integral curve of $T$ gives a foliation of $\Sigma$ hence $\Sigma$ splits as $\Gamma\times\mathbb{R}$ where the Euclidean factor is parallel to $T$ direction and $\Gamma$ is minimal geodesic in $\mathbb{R}^2$. Therefore $\Sigma$ has to be a totally geodesic plane parallel to the $T$ direction. 

Next, we consider the case when $H>0$ everywhere. We write $T^\perp$ as a plane with unit normal $T$ passing through the origin. From $H=\langle\nu, T\rangle>0$, $\Sigma$ is graphical over a domain $\Omega$ inside the plane $T^\perp$. By a $SO(n+1)$ rotation, we can assume that $T^\perp=\mathbb{R}^2\times\{0\}$. From \cite[Corollary 4.3]{shahriyari2015}, $\Omega$ can only be the entire plane $\mathbb{R}^2$, half plane or slab $\mathbb{R}\times (-b,b)$ with $b\geq\frac{\pi}{2}$. Moreover, we can also assume that $\Sigma$ is convex since a result of \cite[Theorem 1.1]{spruckxiao2020} says that complete mean convex translating surface has to be convex (but not necessarily strictly convex). We now divide into three cases and discuss each of them separately.

Case(i): $\Omega$ is $\mathbb{R}^2$, then $\Sigma$ is an entire convex graphical translator over $T^\perp$. Wang \cite{wang2011} shows that such $\Sigma$ must be a bowl soliton which is a graph over $T^\perp$. However bowl soliton cannot be entire graphical over two transversal planes at the same time. It is a contradiction.

Case(ii): $\Omega$ is half plane. According to \cite[Theorem 6.7]{HIMW2019} there exist no complete graphical translator over half plane. Hence $\Omega$ cannot be a half plane.

Case(iii): The remaining case is when $\Omega$ is a slab $\mathbb{R}\times (-b,b)$ for some $b\geq\frac{\pi}{2}$. By convexity of $\Sigma$, its Gauss curvature satisfies $K\geq 0$. Using \cite[Theorem 2.2]{HIMW2019}, if $K=0$ at some point then it vanishes everywhere and $\Sigma$ can only be grim reaper surface or tilted grim reaper surface (as a translating graph over strip in $T^\perp$ plane) both of which are not entire graphical over $xy$-plane (see also \cite[Theorem 2.7]{martin2015}). Hence we can assume $K>0$ everywhere or equivalently $\Sigma$ is strictly convex. The classification result of Hoffmann et al \cite[Theorem 7.1]{HIMW2019} on complete strictly convex translator inside a slab tells us that $\Sigma$ can only be $\triangle$-wing, grim reaper surface or tilted grim reaper surface (as a translating graph over slab in $T^\perp$ plane) in which none of them are entire graphical over $xy$-plane. 

From the three cases above, we conclude that $H>0$ cannot occur and $\Sigma$ has to be minimal. 
\end{proof}

Combining our result above and the rigidity result of Spruck-Xiao \cite{spruckxiao2020}, we obtain the following rigidity result for entire graphical translator with arbitrary translating direction:
\begin{corollary}
Let $\Sigma^2\subseteq\mathbb{R}^3$ be an entire graphical translator with any unit translating velocity $T$. If $H\geq 0$ then $\Sigma$ is either a plane parallel to the $T$ direction or a bowl soliton.
\end{corollary}

Using our result and a classification result by Martin et al., we can also obtain the following immediate consequence 
\begin{corollary}
Let $\Sigma^2\subseteq\mathbb{R}^3$ be an entire graphical translator with unit translating velocity $T$ satisfying $0\leq\langle T,e_3\rangle<1$. If $H\geq 0$ outside a compact subset of $\Sigma$ then $\Sigma$ is a plane.
\end{corollary}
\begin{proof}
We only need to show that $\{H=-1\}=\emptyset$ then \cite[Theorem D]{martin2015} implies that $\Sigma$ is either a plane or $H>0$. However the later cannot happen by Theorem \ref{mean convex}. 

Suppose there exist $p\in\Sigma$ such that $H(p)=-1$, by (\ref{translator eqn}) we have $\nu(p)=-T$ where $\nu$ is upward unit normal. Then $\langle\nu,e_3\rangle=-\langle T,e_3\rangle\leq 0$, contradict with the graphical assumption of $\Sigma$. 
\end{proof}

\begin{remark}
In \cite[Theorem D]{martin2015}, Martin et al. show that a translating soliton $\Sigma$ satisfying $H>-1$ everywhere and $H\geq 0$ outside a compact set is either isometric to a plane or $H>0$ everywhere. Without further assumption, it is not possible to exclude the case for $\langle\nu,T\rangle=H>0$ since a priori a rotation of non-planar vertical graphical translator could exists. On the other hand, with additional graphical condition $\langle\nu,e_3\rangle>0$,  we can rule out the case for $H>0$. \end{remark}

\section{Graphical translator with small entropy}\label{section small entropy}
In this section, we study the rigidity problem of non-vertical graphical translating soliton when its entropy is less than 2.
\begin{theorem}[Theorem \ref{thm rigidity in n=2 assuming lambda <2}]
Let $\Omega$ be an open subset of $\mathbb R^2$ and let 
$$\Sigma  = \{ (x, y, u(x, y)) : (x, y)\in \Omega\}$$
be a complete graphical translating soliton in the direction $T$ not parallel to $e_3$. If $\lambda (\Sigma)<2$, $\Sigma$ is a hyperplane. 
\end{theorem}
\begin{proof}
For all $t<0$, let $\Sigma_t = \Sigma+ tT$. For any $\delta>0$, consider the re-scaling $\Sigma^\delta_t = \delta \Sigma_{\delta^{-2} t}$. As discussed in section \ref{Prelimenary}, there is $\delta_i \to 0$, $\tau_i \to -1$ so that $\{ \Sigma^{\delta_i}_{\tau_i}\}$ converges locally smoothly to an embedded self-shrinker $\Sigma_\infty$, possibly with multiplicity. Moreover, $\lambda (\Sigma_\infty) \le \lambda (\Sigma)<2$. 

Since $\Sigma$ is graphical, it is simply connected and thus $\Sigma_\infty$ is a genus $0$ embedded self-shrinker. By the classification of Brendle \cite{Brendle}, $\Sigma_\infty$ is either the round sphere of radius $2$, the cylinder of radius $\sqrt 2$ or the plane. Since $\Sigma$ is non-compact, $\Sigma_\infty$ cannot be the round sphere. Since $\lambda (\Sigma_\infty)<2$, we conclude that $\Sigma_\infty$ is either the multiplicity one cylinder or the multiplicity one plane. In particular, the convergence $\Sigma^{\delta_i} _{\tau_i}\to \Sigma_\infty$ is locally smooth.

If $\Sigma_\infty$ is the multiplicity one plane, than $\lambda (\Sigma) = 1$ and thus $\Sigma$ is itself a hyperplane. 

Next we assume that $\Sigma_\infty$ is the multiplicity one cylinder and derive a contradiction. Since $\Sigma$ is graphical in the $e_3$ direction, $\Sigma^\delta_{t}$ is graphical in the $e_3$ direction for all $\delta >0$ and $t<0$. Since $\Sigma^{\delta_i}_{\tau_i}$ is close to $\Sigma_\infty$ in large balls, the axis of the cylinder $\Sigma_\infty$ must be parallel to $e_3$, that is 
$$\Sigma_\infty = C = \{ (x, y, z) \in \mathbb R^3 : x^2 + y^2 = 2\}.$$
Let $q_0 = cT$, where $c>0$ is fixed so that $q_0$ lies outside of $C$. Let $S$ be the sphere in $\mathbb R^3$ centered at $q_0$ with radius $r_0$. Here $r_0$ is chosen small so that $S$ is disjoint from the cylinder $C$. Let $\epsilon>0$ be the distance between $S$ and $C$. Let $K$ be a fixed large box in $\mathbb R^3$ centered at $(0,0,0)$ and contains $q_0$ and $S$. Since $\Sigma^{\delta_i}_{\tau_i}$ converges smoothly in $K$ to $C$, there is $N\in \mathbb N$ so that if $i\ge N$, $\Sigma^{\delta_i}_{\tau_i}\cap K$ can be represented as a graph of a function $u_i$ on $C\cap K$ and $|u_i|<\epsilon$. By the choice of $\epsilon$, $S$ is disjoint from $\Sigma^{\delta_i}_{\tau _i}$ for all $i\ge N$. Since $\Sigma$ is a translator moving with velocity $T$, $\Sigma^{\delta _i}_t$ is a translating solution moving with velocity $\delta_i^{-2} T$. For any $i\ge N$, there is $q_i \in \Sigma^{\delta_i} _{\tau_i}$ lying in the straight line joining $(0,0,0)$ to $q_0 = cT$. In particular, the translating soliton $\Sigma^{\delta_i}_{\tau_i+ t_i}$ contains $q_0$, where 
$$t_i = \frac{|q_0- q_i|}{|\delta_i^{-2} T|} = \delta_i^2 \frac{|q_0- q_i|}{| T|}\le \delta_i^2 c. $$
Note that the MCF starting at the sphere $S$ collapses to $q_0$ at time $t = r_0^2/4$. Hence if $\delta_i^2 c < r_0^2/4$, the MCF starting at $M^{\delta_i}_{\tau_i}$ and $S$ respectively intersect at some time $\tilde t_i \le t_i$. This contradicts to the avoidance principle for the MCF and thus $\Sigma_\infty$ cannot be the cylinder. 
\end{proof}

\section{Graphical translator with certain growth rate}\label{section graphical with growth rate}
In this section, we consider general dimensional graphical translator satisfying certain growth condition.
\begin{theorem}[Theorem \ref{thm main theorem}]
Assume that $T$ is perpendicular to $e_{n+1}$, and let $\Sigma$ be an entire graphical translator with velocity $T$ given by $\Sigma = \{ (x, u(x)) : x\in \mathbb R^n\}$. If there are positive numbers $C_1, C_2$ so that 
\begin{equation} \label{eqn assume on growth condition of u}
|u(q)| \le C_1  + C_2\sqrt{ |\langle q,T\rangle|}, \ \ \ \forall q\in \mathbb R^n
\end{equation}
then $\Sigma$ is the stationary horizontal plane. 
\end{theorem}

\begin{proof}
By a rotation of $\mathbb R^{n+1}$ fixing the $e_{n+1}$ direction, we may assume that $T = e_n$. For any $q\in \mathbb R^n$, write $q = (\vec q, q_n)$, where $\vec q = (q_1, \cdots, q_{n-1})$. Then (\ref{eqn assume on growth condition of u}) is the same as 
\begin{equation} \label{eqn assume on growth condition of u in terms of q_n}
|u(\vec q, q_n) |\le C_1 + C_2 \sqrt{|q_n|}, \ \ \ \text{for all } (\vec q, q_n)\in \mathbb R^n.
\end{equation}

For any $t\in \mathbb R$, $\Sigma_t = \{ x+tT : x\in \Sigma\}$ is also an entire graph on $\mathbb R^n$, 
$$ \Sigma_t = \{ (q, u_t(q)) : q\in \mathbb R^n\}$$
and $u_t : \mathbb R^n \to \mathbb R$ is given by 
\begin{equation}\label{equ u_t}
u_t(\vec q , q_n) = u(\vec q, q_n - t). 
\end{equation}
Since $\Sigma$ is a translator with velocity $T$, $\{\Sigma_t\}$ is a solution to the mean curvature flow and $(t, q)\mapsto u_t(q)$ satisfies the graphical mean curvature flow equation, 
\begin{equation}
    \frac{\partial u_t}{\partial t}=\sqrt{1+|Du_t|^2}D_i\left(\frac{D_iu_t}{\sqrt{1+|Du_t|^2}}\right)
\end{equation}

For any $q= (\vec q, q_n) \in \mathbb R^n$ with $q_n \neq 0$, consider the graphical mean curvature flow $u_t$ restricted to the domain $\Omega = B_{|q_n|} (q) \times [-|q_n|, |q_n|]\subseteq\mathbb{R}^n\times\mathbb{R}$. Then by (\ref{equ u_t}) and (\ref{eqn assume on growth condition of u in terms of q_n}), we have 
\begin{equation} \label{eqn bounds on sup u_t on Omega}
M:=\sup _{\Omega} |u_t(q)| \le C_1 +C_3\sqrt{|q_n|}, 
\end{equation}
where $C_3 = \sqrt{2}C_2$. Using the gradient estimates for graphical mean curvature flow \cite[Corollary 5.3]{EvansSpruck}, one has 
\begin{equation}
|Du(q)| = |Du_0(q)|\le C\left( 1+ \frac{M}{q_n}\right) e ^{CM^2 \left(q_n^{-2} + q_n^{-1}\right)},
\end{equation}
where $C$ is a dimensional constant. By (\ref{eqn bounds on sup u_t on Omega}) this implies when $q_n\neq 0$, 
\begin{equation} \label{eqn gradient bound}
|Du (q)| \le C_4
\end{equation}
for some constant $C_4$. By continuity of $Du$, (\ref{eqn gradient bound}) holds for all $q\in \mathbb R^n$. Hence the Gauss map of $\Sigma$ lies in $B^{S^n}_\lambda (y_0)$ for some $\Lambda <\pi/2$ and $y_0= (0,\cdots, 0,1)$. By \cite[Theorem 1.1]{bao2014gauss}, $\Sigma$ must be a hyperplane.

%For any $q\in \mathbb R^n$ and $R>0$, consider the graphical mean curvature flow $u_t$ defined on $\Omega_1:= B_R(q) \times [-t,0]$. Using the interior estimates \cite[Corollary 3,2(iii)]{EckerHuisken}, one has 
%$$|A(q)| \le C \left( \frac{1}{R^2} + \frac{1}{t}\right) \sup_{\Omega_1} |Du_t|.$$
%Using (\ref{eqn gradient bound}) and takes $R, t\to +\infty$, one has $|A(q)| = 0$. Since $q\in \mathbb R^n$ is arbitrary, $|A|$ is identically zero and thus $\Sigma$ is a hyperplane. From the assumption (\ref{eqn assume on growth condition of u}), $\Sigma$ is the horizontal plane $\Sigma = \{ x_{n+1} = c\}$. 
\end{proof}
\begin{remark}
For the vertical graphical translator case, by using bowl soliton as a barrier, the growth rate at infinity of $u$ can only be $|u(x)|=O(|x|^\alpha)$ where $\alpha\geq 2$ (see Remark 3.2(a) in \cite{martin2015}). The situation is quite different for non-vertical graphical translator case where we don't have bowl soliton as a barrier.  
\end{remark}
Using Ecker-Huisken interior gradient estimate for graphical solution of MCF, we can also obtain rigidity result for graphical translator with slow gradient growth.  
\begin{theorem}[Theorem \ref{slow gradient growth}]
Let $\Sigma$ be an entire graphical translator given by $\Sigma = \{ (x, u(x)) : x\in \mathbb R^n\}$ with unit velocity $T$ not parallel to $e_{n+1}$ direction. If the gradient of $u$ satisfies the growth rate 
\begin{equation}
    |Du(x)|=o(|x|^{1/4})
\end{equation}
then $\Sigma$ is a hyperplane. 
\end{theorem}
\begin{proof}
We can assume that $T=\cos\theta e_n+\sin\theta e_{n+1}$, for some $\theta\in(\pi/2,\pi/2)$, after a rotation of $\mathbb{R}^{n+1}$ fixing the $e_{n+1}$ direction. Similar as before, for any $t\in\mathbb{R}$,  $\Sigma_t:=\Sigma+tT=\{(x,u_t(x):x\in\mathbb{R}^n\}$ is an entire graphical solution of MCF where $u_t:\mathbb{R}^n\rightarrow\mathbb{R}$ is given by
\begin{equation}
    u_t(x_1,\cdots,x_n)=u(x_1,\cdots,x_n-t\cos\theta)+t\sin\theta
\end{equation}
and the gradient of $u_t$ is 
\begin{equation}
    Du_t(x)=Du(x-te_{n})
\end{equation}
Using Ecker-Huisken interior estimate \cite[Corollary 3.2(ii)]{EckerHuisken} for any $R>0$ and $t=R^2$, we obtain an estimate for the curvature of $\Sigma_t$ in term of $Du_t$ i.e. for any $y\in\mathbb{R}^n$
\begin{equation}\label{ecker huisken estimate}
    \sup_{B_{R/2}(y)}|A|^2(R^2)\leq C\frac{1}{R^2}\sup_{B_R(y)\times [0,R^2]} |Du_t|^4
\end{equation}
From growth rate assumption and $t\in[0,R^2]$, we have
\begin{equation}
|Du_t(x)|=o(|x-te_n|^{1/4})=o(R^{1/2})
\end{equation} 
Also since $\Sigma_t$ is a graphical translating solution, we know that for any $t$
\begin{equation}\label{shifting}
    \sup_{B_{R}(y)}|A|^2(t)=\sup_{B_{R}(y-t\cos\theta e_n)}|A|^2(0)
\end{equation}
Letting $R\rightarrow\infty$ will imply that $\Sigma$ is totally geodesic. We can argue the previous statement by contradiction. Suppose $\Sigma$ is not totally geodesic then there exist $y_0\in\pi(\Sigma)$, the projection of $\Sigma$ to its domain, such that $|A|(y_0)\neq 0$. Now we choose $R$ large enough such that RHS of (\ref{ecker huisken estimate}) is less than $|A|(y_0)/2$ and pick $y=y_0+R^2\cos\theta e_n$. Our choice of $R$, $y$ and also (\ref{shifting}) give us a contradiction to (\ref{ecker huisken estimate}). 
\end{proof}

\section{Counter example in higher dimension}\label{section counter example}
In this section, we give an example to illustrate why the growth condition in Theorem \ref{thm main theorem} is necessary and that Theorem \ref{thm rigidity in n=2 assuming lambda <2} and Theorem \ref{mean convex} does not hold for $n\ge 3$. Let $\mathcal B$ be the bowl soliton in $\mathbb R^3$ given as a graph $\mathcal B= \{ (x, y, b(x, y)) : (x, y)\in \mathbb R^2\}$. Then for each $n\ge 3$, consider the translator $\mathcal B \times \mathbb R^{n-2}$. That is, 
\begin{equation}
\mathcal B\times \mathbb R^{n-2} = \{ ( x, y, z_1, \cdots, z_{n-2}, b(x, y) ) : (x, y, z_1, \cdots, z_{n-2})\in \mathbb R^n\}. 
\end{equation}
Let $R \in SO(n+1)$ be the rotation 
$$R(x, y, z_1, \cdots, z_{n-2}, w) = \left(x, y, z_1, \cdots, z_{n-3}, \frac{1}{\sqrt{2}} (z_{n-2} - w), \frac{1}{\sqrt{2}} (z_{n-2}+w)\right)$$
and let $M = R (\mathcal B\times \mathbb R^{n-2})$. 

Note that $M$ is an entire graph over $\mathbb R^n \times \{0\}$. Indeed, 
\begin{equation}
M = \{ (x, y, z_1, \cdots, z_{n-3}, \tilde z, \tilde z + \sqrt 2 b(x, y)) : (x, y, z_1, \cdots, z_{n-3}, \tilde z) \in \mathbb R^n\}. 
\end{equation}
$M$ is a translator in the direction $T=Re_{n+1} = \frac{1}{\sqrt 2} (-e_n + e_{n+1})$. Since $$\lambda (M) = \lambda (\mathcal B\times \mathbb R^{n-2}) = \lambda (\mathcal B) = \lambda (\mathbb S^1) <2$$
this implies that Theorem \ref{thm rigidity in n=2 assuming lambda <2} does not hold for $n\ge 3$. Since $M$ is convex and thus mean convex, this also implies that the Theorem \ref{mean convex} cannot be generalized to $n\ge 3$. Lastly, since 
$$T_1:= -\sqrt 2e_n = T - \frac{1}{\sqrt 2}(e_n+ e_{n+1}) $$
and $e_n + e_{n+1}$ is tangential to $M$, so 
\[
T_1^\perp=T^\perp
\]
If we consider $\bar{M}=\sqrt{2}M$, since $M$ is a graphical translator moving in $T$ direction then $\bar{M}$ will be a graphical translator moving in $-e_n$ direction.
This also implies that Theorem \ref{thm main theorem} does not hold without extra growth assumption when $n\ge 3$. 

\section*{Acknowledgement}
The first author is supported in part by DFF Sapere Aude 7027-00110B, by CPH-GEOTOP-DNRF151 and by CF21-0680 from respectively the Independent Research Fund Denmark, the Danish National Research Foundation and the Carlsberg Foundation.The second and third author are supported in part by the National Research Foundation of Korea (NRF-2021R1A4A1032418). 

\bibliographystyle{alpha}
\bibliography{reference}
\end{document}